\documentclass[10 pt, a4 paper, twoside]{article}
\usepackage{amsmath,amssymb}
\usepackage{longtable,mathrsfs}
\usepackage{hyperref}
\usepackage{amsmath, amsthm}
\usepackage[mathscr]{euscript}
\newtheorem{thm}{Theorem}

\newtheorem{lem}{Lemma}

\newtheorem{re}{Remark}
\newtheorem{defn}{Definition}
\newcommand{\B}{\mathbb{B}}
\newcommand{\R}{\mathbb{R}}

{ \setlength{\oddsidemargin}{-2mm}
  \setlength{\evensidemargin}{6mm} }
  \newlength{\titleright}
\allowdisplaybreaks
\begin{document}
\title{\vspace*{-6pc}{\bf On existence and approximation of solutions for nonlinear Hilfer fractional differential equations}}

\vspace{1cm}\author{ {\small D. B. Dhaigude$^{1}$
      \footnote{Corresponding author. Email address: sandeeppb7@gmail@.com,  Tel:+91--9421475347}, Sandeep P. Bhairat$^{2}$}\ \\
{\footnotesize \it $^{1,2}$Department of mathematics, Dr. Babasaheb Ambedkar Marathwada University, Aurangabad--431 004 (M.S.) India.}\\}

\date{}

\maketitle

{\vspace*{0.5pc} \hrule\hrule\vspace*{2pc}}

\hspace{-0.8cm}{\bf Abstract}\\
This paper gives the existence and uniqueness results for solution of fractional differential equations with Hilfer derivative. Using some new techniques and generalizing restrictive conditions imposed on considered function, the iterative scheme for uniformly approximating the solution is established. An example is included to show the applicability of our theoretical results.\\

\hspace{-0.8cm}{\it \footnotesize {\bf Keywords:}} {\small Hilfer derivative; Picard iterative technique; singular fractional differential equation.}
\thispagestyle{empty}
\section{Introduction}
From last few decades, the topic of fractional differential equation have been proved to be a spot-light area of research in the field of applied mathematics. Modeling and exact description of various natural phenomena is fruitfully done by many researchers in biology, rheology, control theory and several other branches of science, see the recent monographs \cite{hr,kst,lv,fm,pi,skm}.

The existence and uniqueness of solution of fractional differential equations comprehensively studied using variety of techniques by imposing and/or generalizing different conditions, for example see \cite{dn,kmf,vl,jd2} and references therein. For fractional differential equation containing Hilfer derivative, the existence and uniqueness results can be found in \cite{kmf,db3}. Using fixed point approach and some suitable conditions on the nonlinear function $f,$ results are obtained in weighted space of continuous functions. In general, solving the nonlinear fractional differential equations analytically is not an easy task. Therefore, we used to approximate the solution up to some efficiency \cite{yy}. With the best of our knowledge, the approximation and uniform convergence of the solution of the considered initial value problem (IVP) is not given in \cite{kmf} which is practically more applicable than that of the fixed point approach. In this article, we used some different techniques and obtained the existence and uniqueness results of solution for singular Hilfer fractional differential equation. The computable iterative scheme along with uniform convergence for the solution is also given.

We consider the initial value problem for fractional differential equation
\begin{align}\label{s1}\begin{cases}
D_{a^+}^{\alpha,\beta}x(t)=f(t,x(t)),&\quad 0<\alpha<1,\, 0\leq\beta\leq1,\quad t>{a},\\
\lim_{t\to{a^{+}}}{(t-a)}^{1-\gamma}x(t)=x_0,&\qquad \gamma=\alpha+\beta(1-\alpha),
\end{cases}
\end{align}
where $D_{a^+}^{\alpha,\beta}$ is the Hilfer fractional derivative and $f$ may be singular at $t=a.$ Note that the initial value considered in initial value problem \eqref{s1} is more suitable than that of considered in earlier one in the literature.

The rest of the paper is organised as follows: the next section covers the useful prerequisites. The main results are proved in section 3 followed by an illustrative example. The last section is conclusion part of the article.

\section{Basic tools}
In this section, we collect some basic definitions and a lemma which are important to us in the development of our main results.

Let $-\infty<a<b<+\infty.$ Let $C[a,b], AC[a,b]$ and $C^{n}[a,b]$ be the spaces of continuous, absolutely continuous, $n-$times continuous and continuously differentiable functions on $[a,b],$ respectively. Here $L^{p}(a,b), p\geq1,$ is the space of Lebesgue integrable functions on $(a,b).$ Let the Euler's gamma and beta functions are defined respectively, by
\begin{equation*}
\Gamma(x)=\int_{0}^{+\infty}s^{x-1}e^{-s}ds,\quad \B{(x,y)}=\int_{0}^{1}(1-s)^{x-1}s^{y-1}ds,\quad x>0,y>0.
\end{equation*}
It is well known that $\B(x,y)=\frac{\Gamma(x)\Gamma(y)}{\Gamma(x+y)},$ for $x>0,y>0,$ \cite{kst}.
\begin{defn}\cite{kst}
Let $\Omega=(a,b]$ and $f:(0,\infty)\to\R$ is a real valued continuous function. The Riemann-Liouville fractional integral of a function $f$ of order $\alpha\in{\R}^{+}$ is denoted as $I_{a^+}^{\alpha}f$ and defined by
\begin{equation}\label{d1}
I_{a^+}^{\alpha}f(t)=\frac{1}{\Gamma(\alpha)}\int_{a}^{t}\frac{f(s)ds}{(t-s)^{1-\alpha}},\quad t>a,
\end{equation}
where $\Gamma(\cdot)$ is the Euler's Gamma function.
\end{defn}
\begin{defn}\cite{kst}
Let $\Omega=(a,b]$ and $f:(0,\infty)\to\R$ is a real valued continuous function. The Riemann-Liouville fractional derivative of function $f$ of order $\alpha\in{\R}_{0}^{+}=[0,+\infty)$ is denoted as $D_{a^+}^{\alpha}f$ and defined by
\begin{equation}\label{d2}
D_{a^+}^{\alpha}f(t)=\frac{1}{\Gamma(n-\alpha)}\frac{d^{n}}{dt^{n}}\int_{a}^{t}\frac{f(s)ds}{(t-s)^{\alpha-n+1}},
\end{equation}
where $n=[\alpha]+1,$ and $[\alpha]$ means the integral part of $\alpha,$ provided the right hand side is pointwise defined on $(0,\infty).$
\end{defn}
\begin{defn}\cite{skm} The Caputo fractional derivative $^{C}D_{a^+}^{\alpha}$ of a continuous function $f:(0,\infty)\to\R$ of order $\alpha\in{\R}^{+}$ is given by
\begin{equation}\label{c1}
^{C}D_{a^+}^{\alpha}f(t)=\frac{1}{\Gamma(n-\alpha)}\int_{a}^{t}\frac{f^{(n)}(s)ds}{(t-s)^{\alpha-n+1}},\,\,n-1<\alpha\leq n,
\end{equation}
provided that the right hand side is pointwise defined on $(0,\infty).$
\end{defn}
\begin{defn} \cite{hr} The Hilfer fractional derivative $D_{a^+}^{\alpha,\beta}$ of a function $f\in L^{1}(a,b)$ of order $n-1<\alpha<n$ and type $0\leq\beta\leq1$ is defined by
  \begin{equation}\label{d3}
    D_{a^+}^{\alpha,\beta}f(t)=I_{a^+}^{\beta(n-\alpha)}D^{n}I_{a^+}^{(1-\beta)(n-\alpha)}f(t),
  \end{equation}
where $I_{a^+}^{\alpha}$ and $D_{a^+}^{\alpha}$ are Riemann-Liouville fractional integral and derivative defined by \eqref{d1} and \eqref{d2}, respectively.
\end{defn}
\begin{re}
The Hilfer (generalized Riemann-Liouville) fractional derivative interpolates between the Riemann-Liouville and the Caputo fractional derivatives as
\begin{equation*}
D_{a^+}^{\alpha,\beta}=
\begin{cases}
DI_{a^+}^{1-\alpha}=D_{a^+}^{\alpha}, & \mbox{if } \beta=0, \\
I_{a^+}^{1-\alpha}D= {^{C}D_{a^+}^{\alpha}}, & \mbox{if } \beta=1.
\end{cases}
\end{equation*}
As a consequence of the fact, these two mostly used approaches (Riemann-Liouville and Caputo) turns out to be particular cases of this general (Hilfer) {differential} operator.
\end{re}
The following lemma is of great importance in the proof of our main results.
\begin{lem}\cite{pi} Suppose that $x>0.$ Then $\Gamma(x)=\lim_{m\to+\infty}\frac{m^{x}m!}{x(x+1)(x+2)\cdots(x+m)}.$
\end{lem}
Denote $D=[a,a+h], D_{h}=(a,a+h]$, $E=\{x:|x(t-a)^{1-\gamma}-x_0|\leq b\}$ for $h>0,b>0$ and $t\in{D_h}.$ We say a function $x(t)$ is a solution of IVP \eqref{s1} if there exist $l>0$ such that $x\in C^{0}(a,a+l]$ satisfies the equation  $D_{a^+}^{\alpha,\beta}x(t)=f(t,x)$ almost everywhere on $(a,a+l],$ alongwith the condition $\lim_{t\to{a^{+}}}{(t-a)}^{1-\gamma}x(t)=x_0.$ To construct the main results, let us make the following hypotheses:
\begin{description}
\item[(H1)]  $(t,x)\to f(t,(t-a)^{\gamma-1}x(t))$ is defined on ${D}_{h}\times E$ satisfies:
\begin{itemize}
\item[(i)] $x\to f(t,(t-a)^{\gamma-1}x(t))$ is continuous on $E$ for all $t\in{D_{h}}$,\\
 $t\to f(t,(t-a)^{\gamma-1}x(t))$ is measurable on $D_{h}$ for all $x\in E;$\
\item[(ii)] there exist $k>(\beta(1-\alpha)-1)$ and $M\geq0$ such that the relation $|f(t,(t-a)^{\gamma-1}x(t))|\leq M(t-a)^{k}$ holds for all $t\in D_{h}$ and $x\in E,$
\end{itemize}
\item[(H2)] there exists $A>0$ such that $|f(t,(t-a)^{\gamma-1}x_1(t))-f(t,(t-a)^{\gamma-1}x_2(t))|$ $\leq A(t-a)^{k}|x_1-x_2|,$ for all $t\in (a,a+l]$ and $x_1,x_2\in E.$
\end{description}
\begin{re}
In hypothesis \textbf{(H1)}, if $(t-a)^{-k}f(t,(t-a)^{\gamma-1}x(t))$ is continuous on $D\times E,$ one may choose $M=\max_{t\in[a,a+h]}(t-a)^{-k}f(t,(t-a)^{\gamma-1}x(t))$ continuous on ${D_h\times E}$ for all $x\in E.$
\end{re}
\section{Main results}
In this section, we prove our main results. Let $l=\min\bigg{\{h,{\big(\frac{b}{M}\frac{\Gamma(\alpha)}{\B(\alpha,k+1)}\big)}^{\frac{1}{\mu+k}}\bigg\}},\, \mu=1-\beta(1-\alpha).$
\begin{lem}
Suppose that \textbf{(H1)} holds. Then $x:[a,a+l]\to\R$ is a solution of IVP \eqref{s1} if and only if $x:(a,a+l]\to\R$ is a solution of the integral equation
\begin{equation}\label{s2}
x(t)=x_0(t-a)^{\gamma-1}+\int_{a}^{t}\frac{(t-s)^{\alpha-1}}{\Gamma(\alpha)}f(s,x(s))ds.
\end{equation}
\end{lem}
\begin{proof} Suppose that $x:(a,a+l]\to\R$ is a solution of IVP \eqref{s1}. Then $|(t-a)^{1-\gamma}x(t)-x_0|\leq b$ for all $t\in(a,a+l].$ From \textbf{(H1)}, there exists a $k>(\beta(1-\alpha)-1)$ and $M\geq0$ such that
\begin{equation*}
  |f(t,x(t))|=|f(t,(t-a)^{\gamma-1}(t-a)^{1-\gamma}x(t))|\leq M(t-a)^{k},\quad \text{for all}\quad t\in(a,a+l].
\end{equation*}
Then we have,
\begin{align*}
\bigg{|}\int_{a}^{t}\frac{(t-s)^{\alpha-1}}{\Gamma(\alpha)}f(s,x(s))ds\bigg{|}&\leq \int_{a}^{t}\frac{(t-s)^{\alpha-1}}{\Gamma(\alpha)}M(s-a)^{k}ds\\
&=M(t-a)^{\alpha+k}\frac{\B(\alpha,k+1)}{\Gamma(\alpha)}.
\end{align*}Clearly,
\begin{equation*}
\lim_{t\to a}(t-a)^{1-\gamma}\int_{a}^{t}\frac{(t-s)^{\alpha-1}}{\Gamma(\alpha)}f(s,x(s))ds=0.
\end{equation*}
It follows that
\begin{equation*}
x(t)=x_0(t-a)^{\gamma-1}+\int_{a}^{t}\frac{(t-s)^{\alpha-1}}{\Gamma(\alpha)}f(s,x(s))ds,\quad t\in(a,a+l].
\end{equation*}
Since $k>(\beta(1-\alpha)-1),$ then $x\in{C^{0}(a,a+l]}$ is a solution of integral equation \eqref{s2}.

Conversely, it is easy to see that $x:(a,a+l]\to\R$ is a solution of integral equation \eqref{s2} implies that $x$ is a solution of IVP \eqref{s1} defined on $[a,a+l].$ This completes the proof.
\end{proof}

Choose a Picard function sequence as
\begin{align*}
  \phi_0(t)&=x_0(t-a)^{\gamma-1},\qquad t\in(a,a+l], \\
  \phi_n(t)=\phi_0(t)+&\int_{a}^{t}\frac{(t-s)^{\alpha-1}}{\Gamma(\alpha)}f(s,\phi_{n-1}(s))ds,\quad t\in(a,a+l],\quad n=1,2,\cdots.
\end{align*}
\begin{lem}
Suppose \textbf{(H1)} holds. Then $\phi_n$ is continuous on $(a,a+l]$ and satisfies $|(t-a)^{1-\gamma}\phi_n(t)-x_0|\leq b.$
\end{lem}
\begin{proof} From \textbf{(H1)}, there exist $k>(\beta(1-\alpha)-1)$ and $M\geq0$ such that $|f(t,(t-a)^{\gamma-1}x)|\leq M(t-a)^{k}$ for all $t\in{D_h}$ and $|x(t-a)^{1-\gamma}-x_0|\leq b.$ If we take $n=1,$ we have
\begin{equation}\label{l1}
  \phi_1(t)=x_0(t-a)^{\gamma-1}+\int_{a}^{t}\frac{(t-s)^{\alpha-1}}{\Gamma(\alpha)}f(s,\phi_{0}(s))ds.
\end{equation}
Then
\begin{equation*}
\bigg{|}\int_{a}^{t}\frac{(t-s)^{\alpha-1}}{\Gamma(\alpha)}f(s,\phi_0(s))ds\bigg{|}\leq \int_{a}^{t}\frac{(t-s)^{\alpha-1}}{\Gamma(\alpha)}M(s-a)^{k}ds=M(t-a)^{\alpha+k}\frac{\B(\alpha,k+1)}{\Gamma(\alpha)}.
\end{equation*}
Clearly, $\phi_1\in{C^{0}(a,a+l]}$ and from equation \eqref{l1}, we have
\begin{align}\label{l2}
  |(t-a)^{1-\gamma}\phi_1(t)-x_0|&\leq (t-a)^{1-\gamma}M(t-a)^{\alpha+k}\frac{\B(\alpha,k+1)}{\Gamma(\alpha)}\nonumber\\
  &\leq Ml^{\alpha+k+1-\gamma}\frac{\B(\alpha,k+1)}{\Gamma(\alpha)}.
\end{align}
Further suppose that $\phi_n\in{C^{0}[a,a+l]}$ and $|(t-a)^{1-\gamma}\phi_n(t)-x_0|\leq b$ for all $t\in[a,a+l].$ We obtain
\begin{equation}\label{l3}
  \phi_{n+1}(t)=x_0(t-a)^{\gamma-1}+\int_{a}^{t}\frac{(t-s)^{\alpha-1}}{\Gamma(\alpha)}f(s,\phi_{n}(s))ds.
\end{equation}
From above discussion, we get $\phi_{n+1}(t)\in {C^{0}(a,a+l]}$ and by equation \eqref{l3},
\begin{align*}
   |(t-a)^{1-\gamma}\phi_{n+1}(t)-x_0|&\leq (t-a)^{1-\gamma}\int_{a}^{t}\frac{(t-s)^{\alpha-1}}{\Gamma(\alpha)}M(s-a)^{k}ds\\
  &=M(t-a)^{\alpha+k+1-\gamma}\frac{\B(\alpha,k+1)}{\Gamma(\alpha)} \\
  &\leq Ml^{\alpha+k+1-\gamma}\frac{\B(\alpha,k+1)}{\Gamma(\alpha)}\leq b.
\end{align*}
Thus, the result is true for $n+1.$ By the mathematical induction principle, the result is true for all $n.$ The proof is thus complete.
\end{proof}

\begin{thm}
  Suppose that \textbf{(H1)} and \textbf{(H2)} holds. Then the sequence $\{(t-a)^{1-\gamma}\phi_n(t)\}$ is uniformly convergent on $[a,a+l].$
\end{thm}
\begin{proof} For $t\in[a,a+l],$ consider the series
\begin{equation*}
{(t-a)^{1-\gamma}\phi_0(t)}+{(t-a)^{1-\gamma}[\phi_1(t)-\phi_0(t)]}+\cdots+{(t-a)^{1-\gamma}[\phi_n(t)-\phi_{n-1}(t)]}+\cdots.
\end{equation*}
By relation \eqref{l2} in the proof of Lemma 3,
\begin{equation*}
  (t-a)^{1-\gamma}|\phi_1(t)-\phi_0(t)|\leq M(t-a)^{\alpha+k+1-\gamma}\frac{\B(\alpha,k+1)}{\Gamma(\alpha)}, \qquad t\in [a,a+l].
\end{equation*}
From Lemma 3,
\begin{align*}
(t-a)^{1-\gamma}|\phi_2(t)&-\phi_1(t)|\leq(t-a)^{1-\gamma}\int_{a}^{t}\frac{(t-s)^{\alpha-1}}{\Gamma(\alpha)}|f(s,\phi_1(s))-f(s,\phi_0(s))|ds\\
=&(t-a)^{1-\gamma}\int_{a}^{t}\frac{(t-s)^{\alpha-1}}{\Gamma(\alpha)}|f(s,(s-a)^{\gamma-1}(s-a)^{1-\gamma}\phi_1(s))\\
&\hspace{3cm}-f(s,(s-a)^{\gamma-1}(s-a)^{1-\gamma}\phi_0(s))|ds\\
\leq&(t-a)^{1-\gamma}\int_{a}^{t}\frac{(t-s)^{\alpha-1}}{\Gamma(\alpha)}A(s-a)^{k}|(s-a)^{1-\gamma}\phi_1(s)-(s-a)^{1-\gamma}\phi_0(s)|ds\\
\leq&(t-a)^{1-\gamma}\int_{a}^{t}\frac{(t-s)^{\alpha-1}}{\Gamma(\alpha)}A(s-a)^{k}[(s-a)^{1-\gamma}|\phi_1(s)-\phi_0(s)|]ds\\
\leq&(t-a)^{1-\gamma}\int_{a}^{t}\frac{(t-s)^{\alpha-1}}{\Gamma(\alpha)}A(s-a)^{k}\big[M(s-a)^{\alpha+k+1-\gamma}\frac{\B(\alpha,k+1)}{\Gamma(\alpha)}\big]ds\\
=&AM\frac{(t-a)^{1-\gamma}}{\Gamma(\alpha)}\frac{\B(\alpha,k+1)}{\Gamma(\alpha)}\int_{a}^{t}{(t-s)^{\alpha-1}}(s-a)^{\alpha+2k+1-\gamma}ds\\
=&AM\frac{\B(\alpha,k+1)}{\Gamma(\alpha)}\frac{\B(\alpha,\alpha+2k+2-\gamma)}{\Gamma(\alpha)}(t-a)^{2(\alpha+k+1-\gamma)}.
\end{align*}
Now suppose that
\begin{equation*}
 (t-a)^{1-\gamma}|\phi_{n+1}(t)-\phi_n(t)|\leq A^{n}M(t-a)^{(n+1)(\alpha+k+1-\gamma)}\prod_{i=0}^{n}\frac{\B(\alpha,(i+1)k+i(\alpha+1-\gamma)+1)}{\Gamma(\alpha)}.
\end{equation*}
We have
\begin{align*}
(t-a)^{1-\gamma}|\phi_{n+2}(t)-&\phi_{n+1}(t)|\leq(t-a)^{1-\gamma}\int_{a}^{t}\frac{(t-s)^{\alpha-1}}{\Gamma(\alpha)}|f(s,\phi_{n+1}(s))-f(s,\phi_n(s))|ds\\
=&(t-a)^{1-\gamma}\int_{a}^{t}\frac{(t-s)^{\alpha-1}}{\Gamma(\alpha)}|f(s,(s-a)^{\gamma-1}(s-a)^{1-\gamma}\phi_{n+1}(s))\\
&\hspace{3cm}-f(s,(s-a)^{\gamma-1}(s-a)^{1-\gamma}\phi_n(s))|ds\\
\leq&(t-a)^{1-\gamma}\int_{a}^{t}\frac{(t-s)^{\alpha-1}}{\Gamma(\alpha)}A(s-a)^{k}|(s-a)^{1-\gamma}\phi_{n+1}(s)-(s-a)^{1-\gamma}\phi_n(s)|ds\\
\leq&(t-a)^{1-\gamma}\int_{a}^{t}\frac{(t-s)^{\alpha-1}}{\Gamma(\alpha)}A(s-a)^{k}[(s-a)^{1-\gamma}|\phi_{n+1}(s)-\phi_n(s)|]ds
\end{align*}
\begin{align*}
(t-a)^{1-\gamma}|\phi_{n+2}(t)-\phi_{n+1}(t)|\leq&(t-a)^{1-\gamma}\int_{a}^{t}\frac{(t-s)^{\alpha-1}}{\Gamma(\alpha)}A(s-a)^{k}\bigg[A^{n}M(s-a)^{(n+1)(\alpha+k+1-\gamma)}\\
&\hspace{2.5cm}\times\prod_{i=0}^{n}\frac{\B(\alpha,(i+1)k+i(\alpha+1-\gamma)+1)}{\Gamma(\alpha)}\bigg]ds\\
=&{A^{n+1}M(t-a)^{(n+2)(\alpha+k+1-\gamma)}}\prod_{i=0}^{n+1}\frac{\B(\alpha,(i+1)k+i(\alpha+1-\gamma)+1)}{\Gamma(\alpha)}.
\end{align*}
Thus the result is true for $n+1.$ Thus by mathematical induction, we get
\begin{equation}\label{l4}
(t-a)^{1-\gamma}|\phi_{n+2}(t)-\phi_{n+1}(t)|\leq A^{n+1}Ml^{(n+2)(\alpha+k+1-\gamma)}\prod_{i=0}^{n+1}\frac{\B(\alpha,(i+1)k+i(\alpha+1-\gamma)+1)}{\Gamma(\alpha)}.
\end{equation}
Consider
\begin{equation*}
\sum_{n=1}^{\infty}u_n=\sum_{n=1}^{\infty}MA^{n+1}l^{(n+2)(\alpha+k+1-\gamma)}\prod_{i=0}^{n+1}\frac{\B(\alpha,(i+1)k+i(\alpha+1-\gamma)+1)}{\Gamma(\alpha)}.
\end{equation*}
We obtain
\begin{align*}
\frac{u_{n+1}}{u_n}&=\frac{MA^{n+2}l^{(n+3)(\alpha+k+1-\gamma)}\prod_{i=0}^{n+2}\frac{\B(\alpha,(i+1)k+i(\alpha+1-\gamma)+1)}{\Gamma(\alpha)}}{MA^{n+1}l^{(n+2)(\alpha+k+1-\gamma)}\prod_{i=0}^{n+1}\frac{\B(\alpha,(i+1)k+i(\alpha+1-\gamma)+1)}{\Gamma(\alpha)}}\\
&=Al^{\alpha+k+1-\gamma}\frac{\Gamma((n+3)k+(n+2)(\alpha+1-\gamma)+1)}{\Gamma((n+3)(k+\alpha)+(n+2)(1-\gamma)+1)}.
\end{align*}
By using Lemma 1, we obtain
\begin{align*}
\frac{u_{n+1}}{u_n}&=Al^{\alpha+k+1-\gamma}\frac{\lim_{m\to\infty}\frac{m^{(n+3)k+(n+2)(\alpha+1-\gamma)+1}m!}{((n+3)k+(n+2)(\alpha+1-\gamma)+1)\cdots((n+3)k+(n+2)(\alpha+1-\gamma)+m+1)}}{\lim_{m\to\infty}\frac{m^{(n+3)(k+\alpha)+(n+2)(1-\gamma)+1}m!}{((n+3)(k+\alpha)+(n+2)(1-\gamma)+1)\cdots((n+3)(k+\alpha)+(n+2)(1-\gamma)+m+1)}}
\end{align*}
$\hspace{2.5cm}=Al^{\alpha+k+1-\gamma}[\lim_{m\to\infty}m^{-\alpha}\frac{((n+3)(k+\alpha)+(n+2)(1-\gamma)+1)\cdots((n+3)(k+\alpha)+(n+2)(1-\gamma)+m+1)}
{((n+3)k+(n+2)(\alpha+1-\gamma)+1)\cdots((n+3)k+(n+2)(\alpha+1-\gamma)+m+1)}].$\\ \\
We can see that $\frac{((n+3)(k+\alpha)+(n+2)(1-\gamma)+1)\cdots((n+3)(k+\alpha)+(n+2)(1-\gamma)+m+1)}
{((n+3)k+(n+2)(\alpha+1-\gamma)+1)\cdots((n+3)k+(n+2)(\alpha+1-\gamma)+m+1)}$ is bounded for all $m,n.$ Then $\lim_{n\to\infty}\frac{u_{n+1}}{u_n}=0.$ Thus $\sum_{n=1}^{\infty}u_n$ is convergent.

Hence the series
\begin{equation*}
{(t-a)^{1-\gamma}\phi_0(t)}+{(t-a)^{1-\gamma}[\phi_1(t)-\phi_0(t)]}+\cdots+{(t-a)^{1-\gamma}[\phi_n(t)-\phi_{n-1}(t)]}+\cdots
\end{equation*}
is uniformly convergent for $t\in[a,a+l].$ Therefore the sequence $\{(t-a)^{1-\gamma}\phi_n(t)\}$ is the uniformly convergent sequence on $[a,a+l].$ This completes the proof.
\end{proof}

\begin{thm}
Suppose \textbf{(H1)} and \textbf{(H2)} hold. Then $\phi(t)=(t-a)^{\gamma-1}\lim_{n\to\infty}(t-a)^{1-\gamma}\phi_n(t)$ is a unique continuous solution of integral equation \eqref{s2} defined on $[a,a+l].$
\end{thm}
\begin{proof} Since $\phi(t)=(t-a)^{\gamma-1}\lim_{n\to\infty}(t-a)^{1-\gamma}\phi_n(t)$ on $[a,a+l],$ and by Lemma 3, we can have $(t-a)^{1-\gamma}|\phi(t)-x_0|\leq b.$ Then
\begin{align*}
  |f(t,\phi_{n}(t))-f(t,\phi(t))|\leq A(t-a)^{k}&|\phi_{n}(t)-\phi(t)|,\quad t\in(a,a+l],\\
  (t-a)^{-k}|f(t,\phi_{n}(t))-f(t,\phi(t))|&\leq A|\phi_{n}(t)-\phi(t)|\to0
\end{align*}
uniformly as $n\to\infty$ on $(a,a+l].$ Therefore
\begin{align*}
(t-a)^{1-\gamma}\phi(t)&=\lim_{n\to\infty}\phi_{n}(t)\\
&=x_0+\lim_{n\to\infty}(t-a)^{1-\gamma}\int_{a}^{t}\frac{(t-s)^{\alpha-1}}{\Gamma(\alpha)}(s-a)^{k}((s-a)^{-k}f(s,\phi_{n-1}(s)))ds\\
&=x_0+(t-a)^{1-\gamma}\lim_{n\to\infty}\int_{a}^{t}\frac{(t-s)^{\alpha-1}}{\Gamma(\alpha)}(s-a)^{k}((s-a)^{-k}f(s,\phi_{n-1}(s)))ds\\
&=x_0+(t-a)^{1-\gamma}\int_{a}^{t}\frac{(t-s)^{\alpha-1}}{\Gamma(\alpha)}(s-a)^{k}\lim_{n\to\infty}((s-a)^{-k}f(s,\phi_{n-1}(s)))ds\\
&=x_0+(t-a)^{1-\gamma}\int_{a}^{t}\frac{(t-s)^{\alpha-1}}{\Gamma(\alpha)}f(s,\phi(s))ds.
\end{align*}
Then $\phi$ is a continuous solution of integral equation \eqref{s2} defined on $[a,a+l].$

To prove uniqueness of solution, suppose that $\psi(t)$ defined on $(a,a+l]$ is also solution of integral equation \eqref{s2}. Then $(t-a)^{1-\gamma}|\psi(t)|\leq b$ for all $t\in(a,a+l]$ and
\begin{equation*}
  \psi(t)=x_0(t-a)^{\gamma-1}+\int_{a}^{t}\frac{(t-s)^{\alpha-1}}{\Gamma(\alpha)}f(s,\phi(s))ds,\quad t\in(a,a+l].
\end{equation*}
It is sufficient to prove that $\phi(t)\equiv\psi(t)$ on $(a,a+l].$ From \textbf{(H1)}, there exists a $k>(\beta(1-\alpha)-1)$ and $M\geq0$ such that
\begin{equation*}
  |f(t,\psi(t))|=|f(t,(t-a)^{\gamma-1}(t-a)^{1-\gamma}\psi(t))|\leq M(t-a)^{k},
\end{equation*}
for all $t\in(a,a+l].$ Therefore
\begin{align*}
(t-a)^{1-\gamma}|\phi_{0}(t)-\psi(t)|=&(t-a)^{1-\gamma}\bigg|\int_{a}^{t}\frac{(t-s)^{\alpha-1}}{\Gamma(\alpha)}f(s,\psi(s))ds\bigg|\\
&\leq(t-a)^{1-\gamma}\int_{a}^{t}\frac{(t-s)^{\alpha-1}}{\Gamma(\alpha)}M(s-a)^{k}ds\\
&=M(t-a)^{\alpha+k+1-\gamma}\frac{\B(\alpha,k+1)}{\Gamma(\alpha)}.
\end{align*}
Furthermore, we have
\begin{align*}
(t-a)^{1-\gamma}|\phi_{1}(t)-\psi(t)|=&(t-a)^{1-\gamma}\bigg|\int_{a}^{t}\frac{(t-s)^{\alpha-1}}{\Gamma(\alpha)}[f(s,\phi_0(s))-f(s,\psi(s))]ds\bigg|\\
&\leq AM \frac{\B(\alpha,k+1)}{\Gamma(\alpha)}\frac{\B(\alpha,\alpha+2k+2-\gamma)}{\Gamma(\alpha)}(t-a)^{2(\alpha+k+1-\gamma)}.
\end{align*}
We suppose that
\begin{equation*}
(t-a)^{1-\gamma}|\phi_{n}(t)-\psi(t)|\leq A^{n}M(t-a)^{(n+1)(\alpha+k+1-\gamma)}\prod_{i=0}^{n}\frac{\B(\alpha,(i+1)k+i(\alpha+1-\gamma)+1)}{\Gamma(\alpha)}.
\end{equation*}
Then
\begin{align*}
(t-a)^{1-\gamma}|\phi_{n+1}(t)-\psi(t)|\leq&(t-a)^{1-\gamma}\bigg|\int_{a}^{t}\frac{(t-s)^{\alpha-1}}{\Gamma(\alpha)}[f(s,\phi_n(s))-f(s,\psi(s))]ds\bigg|\\
\leq& A^{n+1}M(t-a)^{(n+2)(\alpha+k+1-\gamma)}\prod_{i=0}^{n+1}\frac{\B(\alpha,(i+1)k+i(\alpha+1-\gamma)+1)}{\Gamma(\alpha)}\\
\leq& A^{n+1}Ml^{(n+2)(\alpha+k+1-\gamma)}\prod_{i=0}^{n+1}\frac{\Gamma((i+1)k+i(\alpha+1-\gamma)+1)}{\Gamma((i+1)(\alpha+k)+i(1-\gamma)+1)}.
\end{align*}
By using the same arguments used in the proof of Theorem 1, we obtain
\begin{equation*}
  \sum_{n=1}^{\infty}A^{n+1}Ml^{(n+2)(\alpha+k+1-\gamma)}\prod_{i=0}^{n+1}\frac{\Gamma((i+1)k+i(\alpha+1-\gamma)+1)}{\Gamma((i+1)(\alpha+k)+i(1-\gamma)+1)}
\end{equation*}
is convergent. Thus $A^{n+1}Ml^{(n+2)(\alpha+k+1-\gamma)}\prod_{i=0}^{n+1}\frac{\Gamma((i+1)k+i(\alpha+1-\gamma)+1)}{\Gamma((i+1)(\alpha+k)+i(1-\gamma)+1)}\to0$ as $n\to\infty.$ We observe that $\lim_{n\to\infty}(t-a)^{1-\gamma}\phi_n(t)=(t-a)^{1-\gamma}\psi(t)$ uniformly on $[a,a+l].$ Thus $\phi(t)\equiv\psi(t)$ on $(a,a+l].$
\end{proof}

\begin{thm}
Suppose that \textbf{(H1)} and \textbf{(H2)} holds. Then the IVP \eqref{s1} has a unique continuous solution $\phi$ defined on $(a,a+l]$ and $\phi(t)=(t-a)^{\gamma-1}\lim_{n\to\infty}(t-a)^{1-\gamma}\phi_{n}(t)$ with
\begin{align*}
  \phi_0(t)&=x_0(t-a)^{\gamma-1},\qquad t\in(a,a+l], \\
  \phi_n(t)=\phi_0(t)&+\int_{a}^{t}\frac{(t-s)^{\alpha-1}}{\Gamma(\alpha)}f(s,\phi_{n-1}(s))ds,\quad t\in(a,a+l],\, n=1,2,\cdots.
\end{align*}
\end{thm}
\begin{proof} From Lemma 2 and Theorem 2, we can easily obtain that $\phi(t)=(t-a)^{\gamma-1}\lim_{n\to\infty}(t-a)^{1-\gamma}\phi_n(t)$ is a unique continuous solution of IVP \eqref{s1} defined on $(a,a+l]$. Thus the proof is ended here.
\end{proof}
\section{An example.}
Consider the following singular fractional IVP
\begin{equation}\label{e}
\begin{cases}
&D_{0+}^{\frac{1}{2},\frac{1}{2}}x(t)=t^{-\frac{1}{3}}[1+t(x(t))^{\frac{4}{3}}],\quad t>{0},\\
&\lim_{t\to{0}}{t^{\frac{1}{4}}}x(t)=3.
\end{cases}
\end{equation}
Let us choose $h=10,b=8.$ In correspondence with IVP \eqref{s1}, we have $\alpha=\frac{1}{2}$, $\beta=\frac{1}{2},\gamma=\frac{3}{4}, a=0$ with $k={-\frac{1}{3}}.$ Here $f(t,x)=t^{-\frac{1}{3}}[1+t(x(t))^{\frac{4}{3}}]$ is singular at $t=0,$ and $x_0=3.$ In the light of Theorem 3, we obtain
\begin{equation*}
M=\max_{t\in[0,10],x\in[-5,11]}[t^{-k}f(t,x(t))]=[1+\sqrt[3]{(11)^{4}}]\approx 25.46
\end{equation*}
and
\begin{align*}
l=\min\bigg\{h,\bigg(\frac{b}{M}\frac{\Gamma(\alpha)}{\B(\alpha,k+1)}\bigg)^{\frac{1}{\mu+k}}\bigg\}=\min\bigg\{10,\bigg(\frac{8}{25.46}\frac{\Gamma(\frac{1}{2})}{\B(\frac{1}{2},\frac{2}{3})}\bigg)^{{\frac{12}{5}}}\bigg\}.
\end{align*}
It follows that all the conditions of Theorem 3 are satisfied. Clearly, the IVP \eqref{e} has a unique solution
\begin{equation*}
\phi(t)=t^{-\frac{1}{4}}\lim_{n\to\infty}t^{\frac{1}{4}}\phi_{n}(t),\quad t\in[0,l]
\end{equation*}
with the choice of Picard function sequence
\begin{align*}
  \phi_0(t)&=x_0t^{-\frac{1}{4}},\qquad t\in(0,l], \\
  \phi_n(t)=3t^{-\frac{1}{4}}+&\int_{0}^{t}\frac{(t-s)^{-\frac{1}{2}}}{\Gamma(\frac{1}{2})}s^{-\frac{1}{3}}[1+(\phi_{n-1}(s))^{\frac{4}{3}}]ds,\quad t\in(0,l],\quad n=1,2,\cdots.
\end{align*}
\begin{re}
The definite interval of solution for which the solution exists could not be determined by the fixed point theory. Whereas, the existence of solution for the singular IVP \eqref{e} is obtained in very short interval $(0,l],$ where $l\approx0.4$ by the Picard iterative scheme. The initial condition of \eqref{e} for $t^{\frac{1}{4}}x(t)$ practically sounds as $t\to0$ with $t^{\frac{1}{4}}x(t)\approx0.8x(t).$
\end{re}

\section{Concluding remarks} The existence and uniqueness of solution for a general class of fractional differential equation is obtained with the help of Picard's successive approximations. The function $f(t,x)$ considered without assuming the  monotonic property and the iterative scheme is developed for approximating the solution. With the help of traditional convergence criteria, the ratio test, the uniform convergence of solution of the considered IVP is established. Our results essentially improves / generalizes the existing results.

\end{document}